\documentclass[11 pt]{amsart}
\usepackage[applemac]{inputenc}

\usepackage[dvipsnames,svgnames,x11names,hyperref]{xcolor}
\definecolor{ffmblue}{HTML}{006092}

\usepackage[hyphens]{url}
\usepackage[colorlinks=true,linkcolor=Maroon,citecolor=Maroon,urlcolor=ffmblue,hypertexnames=false,linktocpage]{hyperref}
\usepackage{bookmark}
\usepackage{amsmath,thmtools,mathtools}
\mathtoolsset{showonlyrefs=true}

\newcounter{mparcnt}

\usepackage{fancyhdr}
\usepackage{esint}
\usepackage{enumerate}

\usepackage{pictexwd,dcpic}
\usepackage{graphicx}
\usepackage{caption}
\swapnumbers
\declaretheorem[name=Theorem,numberwithin=section]{thm}
\declaretheorem[name=Remark,style=remark,sibling=thm]{rem}
\declaretheorem[name=Lemma,sibling=thm]{lemma}

\declaretheorem[name=Definition,style=definition,sibling=thm]{defn}
\declaretheorem[name=Corollary,sibling=thm]{cor}

\numberwithin{equation}{section}

\newcommand{\ti}{\tilde}

\newcommand{\bs}{\backslash}
\newcommand{\cn}{\colon}
\newcommand{\sub}{\subset}
\newcommand{\ov}{\overline}
\newcommand{\mr}{\mathring}

\newcommand{\bbR}{\mathbb{R}}

\newcommand{\bbS}{\mathbb{S}}
\newcommand{\bbH}{\mathbb{H}}

\newcommand{\8}{\infty}

\newcommand{\ga}{\gamma}
\newcommand{\de}{\delta}
\newcommand{\ep}{\epsilon}
\newcommand{\ka}{\kappa}
\newcommand{\la}{\lambda}
\newcommand{\om}{\omega}

\newcommand{\Om}{\Omega}

\newcommand{\Ga}{\Gamma}

\newcommand{\La}{\Lambda}

\newcommand{\cL}{\mathcal{L}}

\newcommand{\cO}{\mathcal{O}}

\newcommand{\del}{\partial}
\newcommand{\n}{\nabla}

\newcommand{\rt}{\sqrt}

\newcommand{\ip}[2]{\left\langle #1,#2 \right\rangle}
\newcommand{\fr}[2]{\frac{#1}{#2}}
\newcommand{\tfr}[2]{\tfrac{#1}{#2}}
\newcommand{\x}{\times}

\DeclareMathOperator{\id}{id}

\DeclareMathOperator{\dist}{dist}

\DeclareMathOperator{\vol}{vol}

\newcommand{\pf}[1]{\begin{proof}#1 \end{proof}}
\newcommand{\eq}[1]{\begin{equation}\begin{alignedat}{2} #1 \end{alignedat}\end{equation}}

\newcommand{\br}[1]{\left(#1\right)}
\newcommand{\abs}[1]{\lvert #1\rvert}
\newcommand{\enum}[1]{\begin{enumerate}[(i)] #1 \end{enumerate}}

\newcommand{\ra}{\rightarrow}
\newcommand{\hra}{\hookrightarrow}

\newcommand{\mrm}{\mathrm}

\newcommand{\hp}{\hphantom}
\newcommand{\q}{\quad}%\usepackage{showkeys}

\begin{document}

\title[Stability of quermassintegral inequalities in $\bbH^{n+1}$]{Stability of the quermassintegral inequalities in hyperbolic space}

\author{Prachi Sahjwani}
\address{\flushleft\parbox{\linewidth}{Cardiff University\\ School of Mathematics\\ Senghennydd Road\\ Cardiff CF24 4AG\\Wales\\ {\href{mailto:sahjwanip@cardiff.ac.uk}{sahjwanip@cardiff.ac.uk}} }}

\author{Julian Scheuer}
\address{\flushleft\parbox{\linewidth}{Goethe-Universit\"at\\ Institut f\"ur Mathematik\\ Robert-Mayer-Str.~10\\ 60325 Frankfurt\\ Germany\\ {\href{mailto:scheuer@math.uni-frankfurt.de}{scheuer@math.uni-frankfurt.de}}}}
\date{\today}
\thanks{This project was funded by a DTP Programme of EPSRC, Project Reference EP/T517951/1, and in particular within the sub-project ``Stability in physical systems governed by curvature quantities'', Project Reference 2601534.}

\begin{abstract}
For the quermassintegral inequalities of horospherically convex hypersurfaces in the $(n+1)$-dimensional hyperbolic space, where $n\geq 2$, we prove a stability estimate relating the Hausdorff distance to a geodesic sphere by the deficit in the quermassintegral inequality. The exponent of the deficit is explicitly given and does not depend on the dimension. The estimate is valid in the class of domains with upper and lower bound on the inradius and an upper bound on a curvature quotient. This is achieved by some new initial value independent curvature estimates for locally constrained flows of inverse type.
\end{abstract}

\maketitle

\section{Introduction}
The isoperimetric inequality is a fundamental result in geometry that relates the volume of a region in the Euclidean, or also in some non-flat spaces, to the surface area of its boundary. In the Euclidean setting, among all bounded domains $\Om\sub \bbR^{n+1}$, $n\geq 1$, there holds

\eq{\label{II}
    \br{\fr{\abs{\Om}}{\om_{n+1}}}^{\fr{n}{n+1}}\leq \fr{\abs{\del\Om}}{(n+1)\om_{n+1}}
}
   with equality only when $\Omega$ is a geodesic ball. Here $\om_{n+1}$ is the volume of the $(n+1)$-dimensional unit ball and $\abs{\cdot}$ stands for the Hausdorff measure of the appropriate dimension. Equality in this inequality is attained if and only if $\Om$ is a ball. Hence it is natural to investigate the stability question, namely how close is $\Omega$ to a geodesic ball, provided the deviation in \eqref{II} from the equality case is small. For the isoperimetric inequality this question has been addressed to great extent, e.g. \cite{FigalliMaggiPratelli:10/2010,FuscoMaggiPratelli:/2008,Ivaki:10/2014} and we are not attempting a more detailed overview here.

The quermassintegral inequalities are a generalisation of the isoperimetric inequality. They are a collection of geometric inequalities that interrelate the coefficients in the Steiner formula, which is the Taylor expansion of the volume of outer parallel bodies of a convex body $K\sub\bbR^{n+1}$,
\eq{
\vol(K+\rho B) = \sum_{m=0}^{n+1}\binom{n+1}{m}W_m(K)\rho^m,
}
see \cite[p.~208]{Schneider:/2014}. 

 In the Euclidean space the $W_{m}$ can be expressed as curvature integrals and the corresponding inequalities are written as follows:
\eq{\left(\int_{\partial \Omega} E_{m-1},\right)^{\frac{n-m}{n+1-m}}\leq C\int_{\partial \Omega} E_{m},
}
where $\Omega \subset \mathbb{R}^{n+1}$ is a convex bounded domain and $E_{m}$ is the (normalised) elementary degree $m$ symmetric polynomial of principal curvatures of $\partial \Omega $ as an embedding in $\mathbb{R}^{n+1}$. The convexity assumption was relaxed to $m$-convex and starshaped in \cite{GuanLi:08/2009}.
In the convex class, the stability for the inequalities has been thoroughly investigated for example in \cite{GroemerSchneider:/1991,Schneider:01/1989}, while in the non-convex case the only available result seems to be that of the second author \cite{Scheuer:03/2021}.
The purpose of this paper is the transfer of such investigations into the $(n+1)$-dimensional hyperbolic space, where the quermassintegral inequalities were proved by Wang/Xia for horospherically convex domains \cite[Thm.~1.1]{WangXia:07/2014}, by using a suitable curvature flow. They proved that if $\Omega$ is a bounded smooth $h$-convex (i.e. all principal curvatures are greater than $1$) domain in $\mathbb{H}^{n+1}$, then there holds 
\begin{equation} \label{AF}
W_{m}(\Omega) \geq f_{m}\circ f_{l}^{-1} (W_{l}(\Omega)),\q  0 \leq l < m \leq n.
\end{equation}
Equality holds if and only if $\Omega$ is a geodesic ball. Here $W_{m}$ is the $m^{th}$ quermassintegral in $\bbH^{n+1}$ (see \autoref{Prelim} for the definition), $f_{m}(r)=W_{m}(B_{r})$, and $f_{l}^{-1}$ is the inverse function of $f_{l}$. 
Hu/Li/Wei gave an alternative proof by using a different flow \cite{HuLiWei:04/2022}. We will review their method later, as we are going to use the same flow for our result. In this paper we study the stability of these inequalities in the hyperbolic space. In particular we prove the following result, which controls the Hausdorff distance of an $h$-convex hypersurface in $\bbH^{n+1}$ to a geodesic sphere by the deviation of the inequality \eqref{AF} from the equality case:

\begin{thm}\label{Stability AF}
Let $n\geq 2$, $\Omega \subset \mathbb{H}^{n+1}$ be an $h$-convex domain, and $1\leq m\leq n-1$. Then there exists a constant $C= C\left(n,\rho_{-}(\Om),\max_{\del\Om}E_{m}/E_{m-1}\right)$ and a geodesic sphere $S_{\bbH}$ such that
\eq{\label{eq:Stability AF}
\mathrm{dist}(\partial \Omega, S_{\mathbb{H}}) \leq C \left(W_{m+1}(\Omega)-f_{m+1}\circ f_{m}^{-1}\left(W_{m}(\Omega)\right)\right)^{\frac{1}{m+2}}.
}
\end{thm}     

Here $\rho_{-}(\Om)$ is the inradius of the domain $\Om$.
The dependence of $C$ on $\rho_{-}(\Om)$ means that we neither control $C$ when $\rho_{-}(\Om)$ tends to zero, nor when it tends to infinity.

\begin{rem}
\enum{
\item Note that the curvature dependence of $C$ does allow for curvature blowup in a certain sense. Namely, the quantity $E_{m}/E_{m-1}$ may remain bounded, even if $\abs{A}^{2}$ becomes unbounded, as can be seen from the example $n-1=m=2$, for which
\eq{\fr{E_{2}}{E_{1}} = c_{n}\fr{\ka_{1}\ka_{2} + \ka_{1}\ka_{3} + \ka_{2}\ka_{3}}{\ka_{1}+\ka_{2} + \ka_{3}}}
remains bounded, unless merely $\ka_{2}$ goes to infinity.
\item Also note that we do not assume $\del\Om$ to be nearly spherical, as it is done for example in the recent paper \cite{ZhouZhou:06/2023}, where the authors a priori assume $W^{2,\8}$ closeness to a sphere and obtain stability of the Fraenkel asymmetry.
}
\end{rem}

In particular, from the previous theorem we get an estimate in terms of $W_2$ and $W_1$ with exponent $1/3$, if we choose $m=1$ and impose a bound on the mean curvature $H = nE_1$. It turns out the under the same assumption we can extend this to arbitrary $m$ with the same exponent.

\begin{thm}\label{Stability AF-2}
Let $n\geq 2$, $\Omega \subset \mathbb{H}^{n+1}$ be an $h$-convex domain, and $1\leq m\leq n-1$. Then there exists a constant $C= C\left(n,\rho_{-}(\Om),\max_{\del\Om}H\right)$ and a geodesic sphere $S_{\bbH}$ such that
\eq{\label{eq:Stability AF-2}
\mathrm{dist}(\partial \Omega, S_{\mathbb{H}}) \leq C \left(W_{m+1}(\Omega)-f_{m+1}\circ f_{m}^{-1}\left(W_{m}(\Omega)\right)\right)^{\frac{1}{3}}.
}
\end{thm}  

The idea of the proof combines two major inputs drawn from different directions. The first one, which is also deeply involved in the actual proof of the quermassintegral inequalities \eqref{AF}, is the use of a suitable curvature flow to be defined later, which preserves $W_{m}(\Om)$ and decreases $W_{m+1}(\Om)$. The flow exists for all times and converges to a geodesic sphere. This proves the inequality. To characterize the equality case, it is observed that $W_{m+1}(\Om)$ is only strictly decreasing, when the traceless second fundamental form is not zero. For the proof of \eqref{AF} this was sufficient, but for the proof of \eqref{eq:Stability AF} we will make this quantitative and obtain an estimate on the traceless second fundamental form. The second input is  an estimate relating the Hausdorff distance to a geodesic sphere with the traceless second fundamental form. Such an estimate, in the form in which we need it, is due to De-Rosa/Gioffr\'e \cite{De-RosaGioffre:/2019}. Combination of these two ingredients will complete the proof.

After reviewing preliminaries in \autoref{Prelim}, we prove new a priori estimates for the locally constrained flow of $h$-convex hypersurfaces in \autoref{sec:estimates}, which are of independent interest. In \autoref{sec:proof} we complete the proof.

\subsection*{Acknowledgments}
We would like to thank Federica Dragoni and Nicolas Dirr for their support and encouragement.

\section{Preliminaries} \label{Prelim}
To study the curvature flow which is used to prove the quermassintegral inequality and their stability, it is useful to view the pointed hyperbolic space $\mathbb{H}^{n+1}$ as the warped product manifold, coming from polar coordinates around a given origin $o$, 
\eq{
\mathbb{H}^{n+1}\bs\{o\}= (0,\infty) \times \mathbb{S}^{n},
}
equipped with the metric
\eq{
\Bar{g} =dr^{2} +\lambda^{2}(r)g_{\mathbb{S}^{n}},
}
where $\lambda(r)= \sinh(r)$ and $g_{\mathbb{S}^{n}}$ is the standard round metric on the $n$-dimensional unit sphere.
We will also occasionally write $\ip{\cdot}{\cdot}$ for $\bar g$. In this paper, $\mrm{d}_{\bbH^{n+1}}$ will always denote the geodesic distance of two points in hyperbolic space, while 
\eq{\dist(K,L) = \inf\{\de>0\cn K\sub B_{\de}(L)~\wedge~L\sub B_{\de}(K)\}
} denotes the Hausdorff distance of two compact sets.

% We define
%\eq{\Phi(r)= \int_{0}^{r} \lambda(s)ds = \cosh(r)-1.
%}
The vector field $\lambda \partial_{r}$ on $\mathbb{H}^{n+1}$ is a conformal Killing field, i.e. 
\eq{\overline{\nabla}(\lambda \partial_{r})= \lambda'\bar{g},}
where $\ov \n$ is the Levi-Civita connection of $\bar g$.

Let $M$ be a smooth closed hypersurface in $\mathbb{H}^{n+1}$ with outward unit normal $\nu$, 
%such that it can be represented as a radial graph in spherical coordinates $(r(\theta), \theta)$ in $\mathbb{H}^{n+1}$:
%\eq{
%M= \{(r(\theta), \theta) \in \mathbb{R}^{+}\times \mathbb{S}^{n} \lvert    \hspace{1mm}\theta \in \mathbb{S}^{n}\}.
%}
%We define a function $\psi\cn \mathbb{S}^{n} \to \mathbb{R}$ such that 
%\eq{D_{i}\psi= \frac{D_{i}r}{\lambda(r)},} 
%where $D$ is the Levi-Civita connection on $\mathbb{S}^{n}$ with respect to the metric $g_{\mathbb{S}^{n}}$.  We define the quantity 
%\eq{
%v= \sqrt{1+\lvert D\psi\rvert^{2}}.
%}
then we define the support function of the hypersurface by
\eq{
u=\langle \lambda(r)\partial_{r}, \nu \rangle.%= \frac{\lambda}{v}.
}
%The following lemmas hold for smooth hypersurfaces $M$ in $\mathbb{H}^{n+1}$ (see \cite{GuanLi:/2015} for more details), 
Writing $(g_{ij})$ for the metric induced on $M$ with inverse $(g^{ij})$ and Levi-Civita connection $\n$, $h_{ij}$ the second fundamental form and $A=(h^{i}_{j}) = (g^{ik}h_{kj})$ the Weingarten operator, we have the following equation, which follows from the conformal Killing property and the Weingarten equation:

\eq{\label{support equation}
    \nabla_{i}u = \langle \la\del_{r}, e_{k}\rangle h_{i}^{k}, %\hspace{3mm} \nabla_{j}\nabla_{i}u= \langle V, \nabla h_{ij}\rangle + \lambda' h_{ij} -uh_{ik}h^{k}_{j}, 
}
where ${e_{1}, \cdots, e_{n} }$ is a basis of the tangent space of $M$.

Using the change of variables,
\eq{
    r= \log(2+\rho)-\log(2-\rho), \hspace{4mm} \rho \in (-2,2),
}
we obtain
\eq{\label{eq:conf}
    \bar g= e^{2\phi} \left(d\rho^{2}+\rho^{2}g_{\mathbb{S}^{n}}\right)\equiv e^{2\phi}\ti g,
}
where 
\eq{   e^{2\phi}= \frac{16}{(4-\rho^{2})^{2}}.
}
As a result, the hyperbolic space can now be viewed as a conformally flat space. We will need a simple lemma about the surface area of a submanifold of $\bbH^{n+1}$, when viewed as a Euclidean submanifold.

\begin{lemma}\label{area-conf}
Let $(M,g)$ be the embedding of a compact smooth manifold $M$ into $\bbH^{n+1}$ with 
\eq{\max_{M}r\leq \La_{0}.}
Then the Euclidean conformal image $\ti M$ in $B_{2}(0)$ as in \eqref{eq:conf} satisfies
\eq{\fr{1}{C}\abs{\ti M}\leq \abs{M}\leq C\abs{\ti M}, }
with $C = C(\La_{0})$.
\end{lemma} 

\pf{
We have with some local parametrisation $X\cn U\ra M$, 
\eq{\abs{M} = \int_{U}\rt{\det g_{ij}} = \int_{U}e^{n\phi}\rt{\det \ti g_{ij}} = \int_{\ti M}e^{n\phi}}

}

The notion of {\it{convexity by horospheres}} or short {\it $h$-convexity}, is crucial for our result:
\begin{defn}
A smooth bounded domain $\Omega \subseteq \mathbb{H}^{n+1}$ is said to be $h$-convex, if the principal curvatures of the boundary $\partial \Omega$ satisfy $\kappa_{i} \geq 1$ for all $i=1,\cdots,n$. Then we also call $\del\Om$ $h$-convex.
\end{defn}

Such $h$-convex domains already enjoy a quite rigid geometry, and several of their geometric quantities are already controlled by the inradius: Let $\rho_{-}(\Om)$ be the inradius of $\Om$, i.e. the largest number, such that a ball of radius equal to that number fits into $\Omega$. Let $o$ be the center of that ball. In \cite[Thm.~1]{BorisenkoMiquel:/1999} it is shown that then 
\eq{\label{C0}\max_{\del \Om}r=\max_{x\in \del\Om}\mrm{d}_{\bbH^{n+1}}(o,x)\leq \rho_{-}(\Om) + \log 2.}
%They also provide an estimate for the support function, namely we have
%\eq{u = \la(r)\ip{\del_{r}}{\nu}\geq \fr{2\sinh(r)\rt{\tanh(r/2)}}{1+\tanh(r/2)}\geq c(\rho_{-}(\Om)).}
Furthermore one can extract an estimate on the support function. Due to \eqref{support equation}, where $u$ attains a minimum, $\n r$ must be zero, since $A$ is invertible. However, $\min_{\del\Om}r = \rho_{-}(\Om)$
and hence
\eq{\label{C1}\min_{\del\Om} u = \min_{\del\Om}\la(r) = \la(\rho_{-}(\Om)).}
The $h$-convexity of a hypersurface of $\bbH^{n+1}$ translates into convexity of the conformal image:

\begin{lemma}\label{convex-conf}
Let $(M,g)$ be an $h$-convex hypersurface of $\bbH^{n+1}$. Then its conformal Euclidean image $\ti M$ in $B_{2}(0)$ as in \eqref{eq:conf} is convex.
\end{lemma}

\pf{We have
\eq{\label{eq:conf-A}e^{\phi}h^{i}_{j} = \ti h^{i}_{j} + d\phi(\ti \nu)\de^{i}_{j},}
see \cite[Equ.~(1.1.51)]{Gerhardt:/2006}. There holds
\eq{\phi = \log e^{\phi} = \log 4 - \log(4-\rho^{2})  }
and hence
\eq{d\phi = \fr{2\rho}{4-\rho^{2}}d\rho,}
which implies
\eq{\ti h^{i}_{j} \geq  \fr{4}{4-\rho^{2}}h^{i}_{j} - \fr{2\rho}{4-\rho^{2}}\de^{i}_{j}\geq \fr{4-2\rho}{4-\rho^{2}}\de^{i}_{j} = \fr{2}{2+\rho}\de^{i}_{j}.}
Hence the second fundamental form is positive definite.
}

Now we define the hyperbolic quermassintegrals.
For any smooth body $\Omega$ in the hyperbolic space $\mathbb{H}^{n+1}$ with boundary $M= \partial \Omega$, the $k^{th}$ quermassintegrals $W_{k}$ is defined inductively as follows:
\eq{
W_{k+1}(\Omega)= \frac{1}{n+1}\int_{M}E_{k}(\kappa)d\mu - \frac{k}{n+2-k}W_{k-1}(\Omega), \hspace{4mm} k=1, \cdots, n-1,}
where
\eq{W_{0}(\Omega) = \abs{\Om},\q
W_{1}(\Omega)= \frac{1}{n+1}\lvert M\rvert.
}
Here $E_{k}$ is the normalized elementary symmetric polynomial in $n$-variables $\kappa= (\kappa_{1}, \cdots, \kappa_{n})$, 
\eq{E_{k}(\kappa)= \frac{1}{\binom{n}{k}}\sum_{1 \leq i_{1}\leq \cdots\leq i_{k} \leq n} \kappa_{i_{1}}\cdots\kappa_{i_{k}}.
}

In this paper, we use the curvature functions
\eq{F(\ka_{i}) = \fr{E_{m}}{E_{m-1}},\q 1\leq m\leq n-1. }
For us, only the properties on the positive cone $\Ga_{+}\sub\bbR^{n}$ matter, where these functions are monotone, i.e.
\eq{\fr{\del F}{\del\ka_{i}}>0}
and concave. We may also understand these functions as being defined on the Weingarten operator, or on the second fundamental form and the metric,
\eq{F = F(\ka) = F(h^{i}_{j}) = F(g,h).}
Then we write
\eq{F^{ij} = \fr{\del F}{\del h_{ij}}}
and there holds
\eq{F^{i}_{j} = \fr{\del F}{\del h^{j}_{i}} = g_{kj}F^{ik}. }
We refer to \cite[Ch.~2]{Gerhardt:/2006} for a thorough treatment.

\section{New a priori estimates for the locally constrained flow}\label{sec:estimates}

Wang/Xia \cite{WangXia:07/2014} proved the quermassintegral inequalities \eqref{AF} in the hyperbolic space by using the following flow:
Let $M_{{0}}= \partial \Omega$ be a smooth, $h$-convex hypersurface in $\mathbb{H}^{{n+1}}$ with 
\eq{
X_{0}\cn \bbS^{n}\ra M_{0} \hra \mathbb{H}^{n+1}.
}
Then the flow is defined as
\eq{
X\cn \bbS^{n} \times [0,\8) &\to \mathbb{H}^{n+1}\\
\frac{\partial }{\partial t}X(\xi,t)&= \left(c(t)- \left(\frac{E_{k}}{E_{l}}\right)^{\frac{1}{k-l}}(x,t)\right) \nu(\xi,t)\\
			X(\cdot,0)&=X_{0},
}
where $\nu$ is the outward normal to the hypersurface, and $c(t)$ is chosen such that the $l^{th}$ quermassintegral is preserved under this flow. 

The same inequality \eqref{AF} was proved by Hu/Li/Wei \cite{HuLiWei:04/2022} where they used a different flow:
\begin{equation} \label{flow}
\begin{split}
\frac{\partial }{\partial t}X(\xi,t) &= \left(\frac{ \lambda'(r)}{F}-u\right)\nu(\xi,t)\\
X(\cdot, 0) &= X_{0},
\end{split}
\end{equation}
with the notation from \autoref{Prelim}.  This flow preserves the $m^{th}$ quermassintegral $W_{m}(\Omega_{t})$ and decreases $W_{m+1}(\Omega_{t})$ monotonically.

We will quantify the proofs from \cite{GuanLi:/2021} and \cite{HuLiWei:04/2022} and employ the flow \eqref{flow} to extract information on the size of the traceless second fundamental form. To exploit this further, we will use the result from De Rosa/Gioffr{\`e}'s paper \cite{De-RosaGioffre:/2019}.  The closeness of the hypersurface to a geodesic sphere can be controlled by the $L^{p}$ norm of the traceless second fundamental form $\mathring{A}$, whenever $\mathring{A}$ is small. Their result is only for the Euclidean space, however we note that up to a term coming from the conformal factor, the traceless second fundamental form is conformally invariant, and hence the umbilicity in the Euclidean and the hyperbolic space are comparable. We will point out the necessary details whenever appropriate. We will also need some refined curvature estimates, which do not depend on their initial values. Therefore we require some evolution equations and additional a priori estimates, which we develop in the sequel.

It is known that the flow \eqref{flow} has arbitrary spheres as barriers, i.e.
 for all $(t,\xi)\in [0,\8)\x \bbS^{n}$ there holds due to \eqref{C0}, 
   \eq{\label{eq:bound-1}\rho_{-}(\Om) = \min_{\del\Om}r \leq r(\xi,t)\leq \max_{\del\Om}r\leq \rho_{-}(\Om) + \log 2.}
   Since the flow preserves the $h$-convexity, we also obtain a uniform $C^{1}$-bound via
   \eq{\label{eq:flow-C1}\la(\rho_{-}(\Om))\leq u(\xi,t)\leq \la(r(\xi),t)\leq \la(\rho_{-}(\Om)+\log 2)\leq e^{\rho_{-}(\Om)}.}

          Let us define the operator
\eq{\label{differential operator}
\mathcal{L}= \partial_{t}- \frac{\lambda'}{F^{2}}F^{ij}\n^{2}_{ij}- \langle \lambda\partial_{r}, \nabla^{(\cdot)} \rangle.
}

\begin{lemma}\label{lemma:evol-g-h}
Along the flow \eqref{flow}, the induced metric $g= (g_{ij})$ and second fundamental form $(h_{ij})$ satisfy the following equations, see \cite[Lemma~3.1]{HuLiWei:04/2022}
\eq{\label{evolution_g}
   \del_{t}g_{ij} = 2\left(\frac{\lambda'(r)}{F}- u\right)h_{ij};
}
\eq{\label{evolution h_{ij}}
      \cL h_{ij}&= \frac{\lambda'}{F^{2}} F^{kl,pq}\nabla_{i}h_{kl}\nabla_{j}h_{pq}-\left(\frac{\lambda'}{F}+u\right)g_{ij}-2u(h^{2})_{ij} \\
      	 &\hp{=}+\frac{1}{F}\left(\nabla_{j}F\nabla_{i}(\frac{\lambda'}{F})+\nabla_{i}F\nabla_{j}(\frac{\lambda'}{F}) \right)\\
      &\hp{=}+ \left(\frac{u}{F}+\lambda' +\frac{\lambda'}{F^{2}}F^{kl}(h_{rk}h^{r}_{l}+g_{kl})\right)h_{ij} .
}
\end{lemma}

%
%\begin{rem}
%    \begin{equation}
%    \begin{split}
%        \mathcal{L}u= \frac{\lambda'u}{F^{2}} F'\circ A^{2}+\lambda'\left(\frac{\lambda'}{F}-u\right) -\frac{1}{F} \langle V, \nabla \lambda'\rangle- \frac{\lambda'}{F}
%    \end{split}
%\end{equation}
%\end{rem}

\begin{lemma}
The curvature function $F$ satisfies
\eq{\cL F & =(1-F^{ij}g_{ij})u+\fr{\lambda'}{F}(F^{2} - F^{ij}(h^{2})_{ij}) +\frac{2}{F}F^{ij}\nabla_{i}F\nabla_{j}\br{\frac{\lambda'}{F}}.
 }
\end{lemma}

\pf{
We use $F = F(h_{ij},g_{ij})$, \autoref{lemma:evol-g-h}  and \cite[Equ.~(2.1.150)]{Gerhardt:/2006} to compute
\eq{\cL F &= F^{ij}\del_{t}h_{ij} + \fr{\del F}{\del g_{ij}}\del_{t}g_{ij} - \fr{\la'}{F^{2}}F^{ij}\n_{ij}F - \ip{\la\del_{r}}{\n F}\\
		&=F^{ij}\del_{t}h_{ij} - 2F^{ik}h_{k}^{j}h_{ij}\br{\fr{\la'}{F}-u}- \fr{\la'}{F^{2}}F^{ij}F^{kl}\n_{kl}h_{ij}\\
		&\hp{=} - \fr{\la'}{F^{2}}F^{ij}F^{kl,rs}\n_{i}h_{kl}\n_{j}h_{rs} - \ip{\la\del_{r}}{\n F}\\
		&=F^{ij}\cL h_{ij} - 2F^{ik}h_{k}^{j}h_{ij}\br{\fr{\la'}{F}-u} - \fr{\la'}{F^{2}}F^{ij}F^{kl,rs}\n_{i}h_{kl}\n_{j}h_{rs}\\
		&=-F^{ij}\left(\frac{\lambda'}{F}+u\right)g_{ij}+\frac{2}{F}F^{ij}\nabla_{i}F\nabla_{j}\br{\frac{\lambda'}{F}}\\
      &\hp{=}+ \left(u+\lambda'F +\frac{\lambda'}{F}F^{kl}(h_{rk}h^{r}_{l}+g_{kl})\right) - 2\fr{\la'}{F}F^{ik}h_{k}^{j}h_{ij}\\
      &=(1-F^{ij}g_{ij})u+\fr{\lambda'}{F}(F^{2} - F^{ij}(h^{2})_{ij}) +\frac{2}{F}F^{ij}\nabla_{i}F\nabla_{j}\br{\frac{\lambda'}{F}}. }
}

\begin{cor}\label{F-bound}
Along the flow \eqref{flow}, the curvature function satisfies the estimate
\eq{1\leq F\leq \max_{t=0}F.}
\end{cor}

\pf{The lower bound follows immediately from the $h$-convexity and the monotonicity of $F$. For the upper bound,
we use the estimates from \cite[Cor.~2.3]{HuLiWei:04/2022}, which give
\eq{F^{2}\leq F^{ij}h_{ik}h^{k}_{j}\leq (n+1-m)F^{2},\q 1\leq F^{ij}g_{ij}\leq m. }
We conclude that at maximal points of $F$ we have $\cL F\leq 0$
and the result follows from the maximum principle.
%We conclude at maximal points for $F$, at which $F\geq 1$,
%\eq{\cL F&\leq -2uF^{2}+(m-1)\fr{\la'}{F}+(n+2-m)\la'F\\
%		&\leq -2\rho_{-}(\Om)F^{2} + (n+1)\la'F\\
%		&\leq -\rho_{-}(\Om)F^{2} + \fr{(n+1)^{2}}{4\rho_{-}(\Om)} \la'^{2}\\
%		&\leq -\rho_{-}(\Om)F^{2} + \fr{(n+1)^{2}}{\rho_{-}(\Om)}\cosh^{2}\rho_{-}(\Om)  ,}
%where we used the concavity of $F$, \eqref{eq:flow-C1} and \cite[Thm.~1]{BorisenkoMiquel:/1999}.
%Elementary ODE comparison gives the result.
}

\begin{lemma}
Along the flow \eqref{flow}, the mean curvature $H = g^{ij}h_{ij}$ evolves as follows.
\eq{
         \mathcal{L}H&=\frac{\lambda'}{F^{2}} F^{kl,pq}\nabla_{i}h_{kl}\nabla^{i}h_{pq}-n\left(\frac{\lambda'}{F}+u\right) -\fr{2\la'}{F^{3}}\abs{\n F}^{2}+\fr{2}{F^{2}}\n_{i}\la'\n^{i}F\\
      &\hp{=}+ \left(\frac{u}{F}+\lambda' +\frac{\lambda'}{F^{2}}F^{kl}(h_{rk}h^{r}_{l}+g_{kl})\right)H - 2\fr{\la'}{F}\abs{A}^{2}.
   }
\end{lemma}

\begin{proof}
Using the evolution of $g_{ij}$, we can easily find the evolution of $g^{ij}$,
\eq{\label{evolution of g^{ij}}
    \frac{\partial}{\partial t}g^{ij}= -2g^{jk}g^{il}\left(\frac{\lambda'}{F}-u\right)h_{kl}.
}
Hence
\eq{\cL H&=g^{ij}\cL h_{ij} - 2\br{\fr{\la'}{F}-u}\abs{A}^{2}\\
		&=\frac{\lambda'}{F^{2}} F^{kl,pq}\nabla_{i}h_{kl}\nabla^{i}h_{pq}-n\left(\frac{\lambda'}{F}+u\right) \\
      	 &\hp{=}+\frac{1}{F}\left(\nabla^{i}F\nabla_{i}(\frac{\lambda'}{F})+\nabla_{i}F\nabla^{i}(\frac{\lambda'}{F}) \right)\\
      &\hp{=}+ \left(\frac{u}{F}+\lambda' +\frac{\lambda'}{F^{2}}F^{kl}(h_{rk}h^{r}_{l}+g_{kl})\right)H - 2\fr{\la'}{F}\abs{A}^{2}\\
      	&=\frac{\lambda'}{F^{2}} F^{kl,pq}\nabla_{i}h_{kl}\nabla^{i}h_{pq}-n\left(\frac{\lambda'}{F}+u\right) -\fr{2\la'}{F^{3}}\abs{\n F}^{2}+\fr{2}{F^{2}}\n_{i}\la'\n^{i}F\\
      &\hp{=}+ \left(\frac{u}{F}+\lambda' +\frac{\lambda'}{F^{2}}F^{kl}(h_{rk}h^{r}_{l}+g_{kl})\right)H - 2\fr{\la'}{F}\abs{A}^{2}.}
%\eq{
%    \frac{\partial}{\partial t}H&=  g^{ij} \bigg[\frac{\lambda'}{F^{2}} F^{kl} \nabla_{k}\nabla_{l}h_{ij} +\frac{\lambda'}{F^{2}} F^{kl,pq}\nabla_{i}h_{kl}\nabla_{j}h_{pq}\\ 
%     &\hp{=}+\frac{1}{F}\left(\nabla_{j}F\nabla_{i}(\frac{\lambda'}{F})+\nabla_{i}F\nabla_{j}(\frac{\lambda'}{F}) \right)+ \langle V, \nabla h_{ij}\rangle \\
%     &\hp{=}+ \left(\frac{u}{F}+\lambda' +\frac{\lambda'}{F^{2}}F^{kl}((h^{2})_{kl}+g_{kl})\right)h_{ij}-2u(h^{2})_{ij}\\
%      &\hp{=} - \left(\frac{\lambda'}{F}+u\right)g_{ij}\bigg] + h_{ij}\bigg[-2g^{jk}g^{il}\left(\frac{\lambda'}{F}-u\right)h_{kl}\bigg]\\
%      &=  \frac{\lambda'}{F^{2}} F^{kl}\nabla_{k}\nabla_{l}H + \frac{\lambda'}{F^{2}}g^{ij}F^{kl,pq}\nabla_{i}h_{kl}\nabla_{j}h_{pq}\\  &\hp{=}+\frac{g^{ij}}{F}\left(\nabla_{j}F\nabla_{i}(\frac{\lambda'}{F})+\nabla_{i}F\nabla_{j}(\frac{\lambda'}{F}) \right)+ \langle V, \nabla H\rangle\\
%      &\hp{=}+ \left(\frac{u}{F}+\lambda'+ \frac{\lambda'}{F^{2}}F^{kl}((h^{2})_{kl}+g_{kl})\right)H-2u(h^{2})_{ij}g^{ij} \\
%      &\hp{=}- \left(\frac{\lambda'}{F}-u\right)-2(h^{2})_{il}g^{il}\left(\frac{\lambda'}{F}-u\right)
%}
%Now we substitute $\frac{\partial}{\partial t}H$ in $\mathcal{L}H$ defined above and get the required result.
\end{proof}

\begin{cor}\label{H-bound}
Along the flow \eqref{flow} and up to time $t=1$, the curvature function satisfies the estimate
\eq{n\leq H\leq \fr{C(n,\rho_{-}(\Om),\max_{M_{0}}F)}{t}.}
\end{cor}

\pf{
We proceed similarly to the proof of \autoref{F-bound}.
At maximal points of $H$ we have, using $\abs{A}^{2}\geq H^{2}/n$ and the concavity of $F$,
\eq{
         \mathcal{L}H&\leq-n\left(\frac{\lambda'}{F}+u\right) -\fr{2\la'}{F^{3}}\abs{\n F}^{2}+\fr{2}{F^{2}}\n_{i}\la'\n^{i}F\\
      &\hp{=}+ \left(\frac{u}{F}+\lambda' +\frac{\lambda'}{F^{2}}F^{kl}(h_{rk}h^{r}_{l}+g_{kl})\right)H - \fr{2}{n}\fr{\la'}{F}H^{2}\\
      &\leq C - \fr{1}{nF}H^{2}\\
      &\leq C - \fr{1}{C}H^{2},
   }
   where in the last step we used \autoref{F-bound}. We have also used Cauchy-Schwarz to absorb $\n F$ and first order terms in $H$. The result again follows from a simple ODE comparison argument.
}

\section{Proof of \autoref{Stability AF} and \autoref{Stability AF-2}}\label{sec:proof}
In this section, we prove \autoref{Stability AF}. In the following proof, we take $C= C\left(n,\rho_{-}(\Om),\max_{\del\Om}F\right)$ to be a generic constant depending on the quantities mentioned.
\begin{proof}
Let $\epsilon> 0$ be such that 
\eq{
W_{m+1}(\Omega) =f_{m+1}\circ f_{m}^{-1}\left(W_{m}(\Omega)\right)+\epsilon.
}
Let $\rho_{-}(\Om)$ be the inradius of $\Om$ and pick the origin $o$ as the center of the corresponding inball.
Under the flow \eqref{flow} with initial surface $\del\Om$, $W_{m+1}(\Om_{t})$ evolves as (see \cite[Prop.~3.1]{WangXia:07/2014} for details) 
\eq{
\frac{\partial }{\partial t}W_{m+1}(\Om_{t})=\fr{n-m}{n+1}\int_{M_{t}}\br{\la'(r)\fr{E_{m-1}}{E_{m}}-u}E_{m+1},
}
where $M_{t} = \del\Om_{t}$.
We compute
\eq{
\label{pf:AF-stability-2}
\int_{0}^{\infty} \int_{M_{t}}\lambda'\left(\frac{E_{m+1}E_{m-1}}{E_{m}}-E_{m}\right)
&= \int_{0}^{\infty}\int_{M_{t}} \left(\frac{\lambda'E_{m-1}}{E_{m}}-u\right)E_{m+1}\\
&= \fr{n+1}{n-m}\int_{0}^{\infty}\frac{\partial}{\partial t} W_{m+1}(\Om_{t})\,dt\\
&=\fr{n+1}{n-m}(W_{m+1}(B)-W_{m+1}(\Omega))\\
&= -\fr{n+1}{n-m}\epsilon.
}   
In the first line of this calculation, we have used the Minkowski formula proved for example in Guan/Li \cite{GuanLi:/2015}
\begin{equation}
\int_{M_{t}} \lambda'(r)E_{m}= \int_{M_{t}}uE_{m+1}.
\end{equation} 
We have also used that $\Omega$ converges to a round ball at infinite time, $\Omega_{\infty}= B$ where \eqref{AF} holds with equality, and $W_{m}$ is preserved under the flow, $W_{m}(B)= W_{m}(\Omega)$.

%Using the Newton-Maclaurin inequality 
%\eq{E_{m}^{2} \geq E_{m-1}E_{m+1},
%}
%and $\lambda'(r) \geq 1 $, we get
%\begin{equation} \label{epsilon-bound}
%\begin{split}
%\int_{0}^{\infty} \int_{M_{t}}\left( E_{m}-\frac{E_{m+1}E_{m-1}}{E_{m}} \right) 
%& \leq \int_{0}^{\infty}\int_{M_{t}}\lambda'(r)\left( E_{m}-\frac{E_{m+1}E_{m-1}}{E_{m}}\right)\\
%&= \epsilon
%\end{split}
%\end{equation}
Along the flow we have
\eq{-\la(r)\leq \fr{\la'(r)}{F}-\fr{\la(r)}{v}\leq \la'(r)}
and hence, using $\la\leq \la'$, 
\eq{
\label{pf:AF-stability-1}\left\lvert \frac{\lambda'(r)}{F}- u \right\rvert \leq \la'(\max_{\del\Om}r)\leq \cosh(\rho_{-}(\Om) + \log 2)\leq 2\cosh(\rho_{-}(\Om)),
}
where we used \eqref{eq:bound-1} and \cite[Thm.~1]{BorisenkoMiquel:/1999}.

Using the above bound, we want to estimate the Hausdorff distance between $M_{t}$ and $M_{0}=\del\Om$. Let $X(\xi,0)$ and $X(\xi,t)$ be two points in $M_{0}$ and $M_{t}$ respectively. Let $\gamma: [0,t] \to \mathbb{H}^{n+1}$ be a curve defined as 
\eq{\gamma(\tau)= X(\xi,\tau).
}
Then we have  due to \eqref{pf:AF-stability-1},
\begin{align*}
\mathrm{d}_{\bbH^{n+1}}(X(0,\xi), X(t,\xi))
& \leq \max_{[0,t]}\lvert \del_{\tau}\gamma \rvert t\leq 2\cosh(\rho_{-}(\Om))t.
\end{align*}
From this, we get 
\begin{equation}
\mathrm{dist}(M_{t}, \del\Om) \leq Ct, \hspace{3mm} \forall t \geq 0.
\end{equation}

From \eqref{pf:AF-stability-2} and $\la'\geq 1$, we get 
\eq{
 \int_{0}^{\8}\int_{M_{t}} \left( E_{m}- \frac{E_{m+1}E_{m-1}}{E_{m}} \right) \leq \fr{n+1}{n-m}\ep.
}
%Let $M^{\epsilon}$ be the hypersurface where the minimum on the left is attained, then
%\begin{equation}
%\mathrm{dist}(M^{\epsilon}, M) \leq C\sqrt{\epsilon}  \hspace{3mm} \text{and} \hspace{3mm} \int_{M^{\epsilon}} \left(E_{m}-\frac{E_{m+1}E_{m-1}}{E_{m}}\right) \leq \sqrt{\epsilon}.
%\end{equation}
Then using \cite[Lemma 4.2]{Scheuer:03/2021} and \autoref{H-bound} we get for $\de>0$,
\eq{\label{pf:Stability AF-3}
\int_{\de}^{2\de}\int_{M_{t}} \lvert \mathring{A}  \rvert^{2} 
		& \leq C\max_{[\de,2\de]} E_{m-1} \int_{\de}^{2\de}\int_{M_{t}} \frac{E_{m+1,n1}^{2}\lvert \mathring{A}\rvert^{2}}{E_{m}}\leq \fr{C}{\de^{m-1}}\ep,
}
%Then using \cite[Lemma 4.2]{Scheuer:03/2021} we get
%\begin{equation*}
%\begin{split}
%\int_{M^{\epsilon}} \lvert \mathring{A}  \rvert^{n+1} 
%& \leq C\int_{M^{\epsilon}} \lvert \mathring{A} \rvert^{2}\\
%& \leq C (\underset{M^{\epsilon}}{\mathrm{max \text{ }}}E_{m}) (\underset{M^{\epsilon}}{\mathrm{min \text{ }}} E_{m+1,n1})^{-1} \int_{M^{\epsilon}} \frac{E_{m+1,n1}^{2}\lvert \mathring{A}\rvert^{2}}{E_{m}}\\
%& \leq C \sqrt{\epsilon}.
%\end{split}
%\end{equation*}
where we also used $E^{2}_{m+1,ij}= \frac{\partial^{2} E_{m+1}}{\partial \kappa_{i}\kappa_{j}}\geq 1$. 
Hence there exists $t_{\de}\in [\de,2\de],$ such that
\eq{\|\mr{A}\|_{L^{2}(M_{t_{\de}})}\leq C\de^{-\tfr{m}{2}}\rt\ep.}

Now put
\eq{\de = \ep^{\fr{1}{m+2}}}
to obtain
\eq{\label{eq:dist}\dist(M_{t_{\de}},\del\Om) + \|\mr{A}\|_{L^{2}(M_{t_{\de}})}\leq C\ep^{\fr{1}{m+2}}.}

%The above computation  gives 
%\begin{equation}\label{secondfundamentalformbound}
%\|{\mathring{A}}\|_{L^{n+1}(M^{\epsilon})} \leq C (\epsilon)^{1/2(n+1)}.
%\end{equation}
In order to apply \cite[Thm.~1.2]{De-RosaGioffre:/2019}, we view $M_{t_{\de}}$ as a Riemannian submanifold of the Euclidean ball of radius $2$, which is conformal to $\bbH^{n+1}$ as in \eqref{eq:conf}. Due to \autoref{convex-conf} and furnishing the Euclidean geometric tensors by a tilde, we see that 
 $\ti M_{t_{\de}}$ is convex. Now we have to normalize $\ti M_{t_{\de}}$,
\eq{\hat M_{t_{\de}} = \br{\fr{\abs{\bbS^{n}}}{\abs{\ti M_{t_{\de}}}}}^{\fr 1n}\ti M_{t_{\de}}\equiv \ga \ti M_{t_{\de}}.}
Note that $\abs{M_{t_{\de}}}$ is controlled from above and below in terms of $\rho_{-}(\Om)$, due to the convergence of the surface area-preserving curvature flow
\eq{\fr{\del}{\del t}X = \br{\fr{\la'}{E_{1}}-u}\nu,}
which converges to a geodesic sphere with radius between $\rho_{-}(\Om)$ and $\rho_{-}(\Om) + \log 2$. Due to \autoref{area-conf} we have $\ga = \ga(n,\rho_{-}(\Om))$.  \cite[Thm.~1.2]{De-RosaGioffre:/2019} gives, provided that $\ep \leq \ep_{0}(n,\rho_{-}(\Om),\max_{\del\Om}F)$ with $\ep_{0}$ sufficiently small, a parametri\-zation
\eq{
\psi\cn \mathbb{S}^{n} &\to \hat M_{t_{\de}}\sub B_{2}(0)\sub \bbR^{n+1}
}
and a point $\cO\in \bbR^{n+1}$, such that
and $f$ satisfies the estimate
\eq{
\|\psi-\id -\cO\|_{W^{2,2}(\mathbb{S}^{n})} \leq C \|\mathring{\hat A}\|_{L^{2}(\hat M_{t_{\de}})}\leq C\|\mr A\|_{L^{2}(M_{t_{\de}})}\leq C\ep^{\fr{1}{m+2}}.
}
This implies that $\hat M_{t_{\de}}$ is Hausdorff-close to the Euclidean unit sphere, that $\ti M_{t_{\de}}$ is close to a Euclidean sphere of radius $\ga^{-1}$ and that in turn $M_{t_{\de}}$ is close to a hyperbolic sphere, with exactly the same error estimate,
\eq{\dist(M_{t_{\de}},S_{\bbH})\leq C\ep^{\fr{1}{m+2}}.}
Employing \eqref{eq:dist} finishes the proof for $\ep\leq \ep_{0}$. However, if $\ep>\ep_{0}$, the estimate is trivial due to 
\eq{\max_{\del\Om}r\leq \rho_{-}(\Om) + \log 2.}

To prove \autoref{Stability AF-2}, we reconvene at \eqref{pf:Stability AF-3} and do not estimate $\max E_{m}$ using \autoref{H-bound}, but the constant itself is now allowed to depend on $H$. Hence the factor $\de^{-m+1}$ is simply not present and in the subsequent computations we can pretend $m$ would be one. The proof can then literally be completed as above.
%This implies
%\begin{align*}
%    \mathrm{dist}_{\tilde{g}}(M^{\epsilon}, \mathbb{S}) & \leq \|e^{f(x)}-e^{0}\|\\
%    &\leq C\|f(x)\|_{W^{2,n+1}(\mathbb{S}^{n})}\\
%    & \leq C \|\mathring{A}\|_{L^{n+1}(M^{\epsilon})}\\
%    & \leq C (\epsilon)^{1/2(n+1)}.
%\end{align*}
%Here, $\textrm{dist}_{\tilde{g}}$ and $\textrm{dist}_{g}$ represents distance obtained with respect to the Euclidean and hyperbolic metric respectively.
%
%This implies 
%\begin{equation*}
%    \mathrm{dist}_{\tilde{g}}(M, \mathbb{S}) \leq C \epsilon^{1/2(n+1)}.
%\end{equation*}
%We want the Hausdorff distance of $M$ with the geodesic sphere in the hyperbolic space. Since every sphere in the hyperbolic space is a sphere in the Euclidean space, we get that under the conformal change, the Hausdorff distance is related as follows:
%\begin{equation*}
%    \mathrm{dist}_{g} (M, \mathbb{S}_{\mathbb{H}}) \leq e^{\phi}_{\mathrm{max}} \mathrm{dist}_{\tilde{g}}(M, \mathbb{S}_{\mathbb{H}}).
%\end{equation*}
%Since $\phi$ is bounded, we get that
%\begin{equation*}
%    \mathrm{dist}_{g}(M, \mathbb{S}_{\mathbb{H}}) \leq C \epsilon^{1/2(n+1)}.
%\end{equation*}
\end{proof}

%
%
%\bibliographystyle{/Users/julianscheuer/Documents/Uni/TexTemplates/shamsplain}
%\bibliography{/Users/julianscheuer/Documents/Uni/TexTemplates/Bibliography.bib}

\providecommand{\bysame}{\leavevmode\hbox to3em{\hrulefill}\thinspace}
\providecommand{\href}[2]{#2}

\end{document}